\documentclass{article}

\usepackage{graphicx} 
\usepackage{amsmath}    
\usepackage{amssymb}    
\usepackage{amsthm}    
\usepackage{dsfont}
\usepackage[a4paper, left=1in, right=1in, top=1in, bottom=1in]{geometry} 
\usepackage{hyperref}
\usepackage{xcolor}

\theoremstyle{plain}
\newtheorem{theorem}{Theorem}
\newtheorem{corollary}[theorem]{Corollary}
\newtheorem{proposition}{Proposition}
\newtheorem{lemma}{Lemma}

\theoremstyle{definition}
\newtheorem{conjecture}{Conjecture}

\newtheorem{definition}{Definition}
\newtheorem{example}{Example}
\newtheorem{question}{Question}
\newtheorem{claim}{Claim}

\title{Heged\"{u}s' Conjecture and Tighter Upper Bounds for Equidistant Codes in Hamming Spaces}
\author{Sihuang Hu, 
        Hexiang Huang and Wei-Hsuan Yu }
\date{}

\begin{document}

\maketitle
\begin{abstract}
An equidistant code is a code in the Hamming space such that two distinct codewords have the same Hamming distance. This paper investigates the bounds for equidistant codes in Hamming spaces. We prove a refined version of Heged\"{u}s’ conjecture about the size of equidistant codes by adjusting the assumption on the Hamming distance \( d \), establishing a new upper bound. Specifically, for any integer \( q \geq 2 \), we demonstrate that the size of \( q \)-ary equidistant codes of length \( n \) is at most \( (q-1)n \), provided \( d \neq \frac{(q-1)n+1}{q} \). Additionally, we extend Deza's theorem from the binary case to the $q$-ary case and then apply this generalization to derive tight upper bounds $\lfloor 2n/d\rfloor$ for sufficiently large $n$ (length of codewords). Amazingly, the upper bound is independent of $q$.
\end{abstract}

\section{Introduction}

Let \( q \geq 2 \) and \( n \) be positive integers. Denote \( [0, q-1] = \{0,1,\ldots,q-1\} \) and \( [n] = \{1,2,\ldots,n\} \). Let \( H_q^n = [0,q-1]^n \) denote the \( q \)-ary \textit{Hamming space}. For $x\in H_q^n$, the $i$-th coordinate of $x$ is denoted by $x(i)$. The \textit{support} of $x$ is
\[\mathrm{supp(x)} = \{i\in [n]:x(i)\neq 0\}.\] 
The \textit{Hamming weight} of a word \( x \in H_q^n \) is the cardinality of its support, namely, \( \mathrm{wt}(x)  = |\mathrm{supp}(x)|\). The \textit{Hamming distance} between two words \( x, y \in H_q^n \) is given by
\[
\mathrm{d_H}(x, y) = |\{i \in [n] : x(i) \neq y(i)\}|.
\]
It is well known that \( \mathrm{d_H} \) defines a metric on \( H_q^n \).  Define the \textit{Johnson space} $J^{n,k}_q$ to the words in $H_q^n$ of weight $k$, namely,
\[J^{n,k}_q = \{x\in H_q^n:\mathrm{wt}(x) = k\}.\]
When $q = 2$, we may drop the subscript and write $J^{n,k}$. Throughout this paper, we identify a subset in $[n]$ with its incidence vector in $\{0,1\}^n$. Thus we can view $J^{n,k}$ as the collection of all $k$-subsets of $[n]$, namely,
\[J^{n,k} = \binom{[n]}{k}.\]

In coding theory, two codes \(C_1\) and \(C_2\) over a finite alphabet are \textit{equivalent}, if one can be obtained from the other by:
\begin{enumerate}
    \item[(a)] A permutation of the coordinates (reordering the positions of symbols in each codeword), and/or
    \item[(b)] A permutation of the symbols within each coordinate (applying a consistent symbol substitution at each position).
\end{enumerate}

A code \( C \subseteq H_q^n \) is said to be an \textit{\( s \)-code} if the pairwise distances between distinct codewords take exactly \( s \) distinct values. 
Particularly, for \( s = 1 \), an \textit{equidistant code} is a code where the Hamming distances between any two distinct codewords are the same.

Equidistant codes play a fundamental role in coding theory and combinatorial geometry. They can be regarded as the Hamming-space analogues of geometric simplices in Euclidean space. 
In particular, binary equidistant codes, such as Hadamard codes (see Chapter 2 of \cite{MacWilliams1977}), are widely employed in communication systems and signal processing due to their efficient encoding and decoding properties, see \cite{Evangelaras2003ApplicationsOH}.

Understanding the maximum possible size of an equidistant code under constraints on distance and alphabet size is a fundamental problem in extremal coding theory. In \cite{Deza1973}, Deza showed that when a binary equidistant code has a sufficiently large cardinality, it must be a trivial code, known as a \textit{$\Delta$-system} or \textit{sunflower}. This phenomenon also appears in other contexts; see, for example, \cite{DezaFrankl1981}, \cite{EtzionRaviv2015}, and \cite{GorlaRavagnani2016}.

Significant progress has been made in the study of equidistant codes, please refer to \cite{FU2003157} and references therein. 
Etzion and Raviv (2015) explored equidistant codes in the Grassmannian, focusing on subspace distances and bounds \cite{EtzionRaviv2015}. Gorla and Ravagnani (2016) investigated the properties of equidistant subspace codes, providing theoretical insights into their structure \cite{GorlaRavagnani2016}. Deza and Frankl (1981) demonstrated that large sets of ternary equidistant vectors form a sunflower structure \cite{DezaFrankl1981}, while Zinoviev and Todorov (2007) constructed maximum $q$-ary equidistant codes for specific distances, such as $d=3$ and $4$ \cite{ZinovievTodorov2007}. These results lay a foundation for further exploration of $q$-ary equidistant codes in various applications.

A classical upper bound for \( s \)-codes in Hamming space was established by Delsarte in \cite{delsarte1973algebraic}, \cite{delsarte1975association}, and later refined by Babai, Snevily, and Wilson \cite{Babai_Snevily_Wilson_1995_A_new_proof} using the polynomial method.

\begin{theorem}[Delsarte \cite{delsarte1975association}, 1975]
    Let \( C \subseteq H_q^n \) be an \( s \)-code. Then
    \[
    |C| \leq \sum_{i=0}^s \binom{n}{i}(q-1)^i.
    \]
\end{theorem}

For \( s = 1 \), we obtain the following.

\begin{corollary}
\label{corollar_upper_bound_for_q_ary_case}
    Let \( C \subseteq H_q^n \) be an equidistant code. Then
    \[
    |C| \leq n(q-1) + 1.
    \]
\end{corollary}

A recent improvement for the binary case was obtained by Heged\"{u}s in \cite{hegedus_2024_a_new_upper}, where a sharper bound was derived for codes with distance \( d \neq \frac{n+1}{2} \).

\begin{theorem}[Heged\"{u}s \cite{hegedus_2024_a_new_upper}, 2024]
    Let \( C \subseteq H_2^n \) be an equidistant code. If the code has distance \( d \neq \frac{n+1}{2} \), then
    \[
    |C| \leq n.
    \]
\end{theorem}

Heged\"{u}s also proposed the following conjecture for the general \( q \)-ary case.

\begin{conjecture}
\label{conjecture}
    Let \( C \subseteq H_q^n \) be an equidistant code. Suppose the code has distance \( d \neq \frac{(q-1)(n+1)}{q} \). Then
    \[
    |C| \leq n(q-1).
    \]
\end{conjecture}

Heged\"{u}s proposed $d\neq \frac{(q-1)n}{q}$ in Conjecture \ref{conjecture}, but we identify that the correct condition should be 
$d\neq \frac{(q-1)n+1}{q}$, as justified by our proofs in Section \ref{section_proof_of_Hegedus_conjecture}. A partial result toward this conjecture was proved by Barg and Musin in \cite{barga_musin_2011_bounds}.

\begin{theorem}[Barg and Musin \cite{barga_musin_2011_bounds}, 2011]
    Let \( C \subseteq H_q^n \) be an equidistant code. If the code has distance \( d \leq \frac{(q-1)n}{q} \), then
    \[
    |C| \leq n(q-1).
    \]
\end{theorem}

In addition to proving a refined version of Heged\"{u}s' conjecture, we also establish upper bounds for equidistant codes that depend on the Hamming distance. 
The proofs utilize a theorem of Deza \cite{Deza1973}.

\begin{theorem}[Deza \cite{Deza1973}, 1973]
\label{Theorem_deza_Delta_system}
    An \((n, k, l)\)-family \(\mathcal{F}\) with cardinality greater than \(k^2 - k + 1\) is a \(\Delta(n, k, l)\)-system, and thus
    \[
    |\mathcal{F}| \leq \frac{n - l}{k - l}.
    \]
\end{theorem}

This theorem can be readily applied to binary equidistant codes, yielding a distance-dependent bound. To derive a similar result for \(q\)-ary equidistant codes, we extend Deza's theorem from the binary case to the \(q\)-ary case.

The paper is structured as follows. In Section \ref{section_preliminaries_on_simplex}, we present preliminaries on equidistant codes in Euclidean space. In Section \ref{section_proof_of_Hegedus_conjecture}, we prove Conjecture \ref{conjecture} with a slight modification to the condition on \( d \). In Section \ref{section_generalization_of_Deza_theorem}, by extending Deza's theorem from the binary case to the \( q \)-ary case, we derive a tight upper bound for equidistant codes that is dependent on Hamming distance. Through this approach, we reveal an intriguing phenomenon: when the code length is sufficiently large, the optimal construction for \( q \)-ary equidistant codes aligns with that of the binary case. Finally, we propose several questions regarding the construction and bounds of equidistant codes in Section \ref{section_discussions}.

\section{Preliminaries on simplex}
\label{section_preliminaries_on_simplex}

Our proofs rely on a technique that maps the alphabet \( [0,q-1] \) to a simplex of size \( q \) in Euclidean space. In this section, we present some fundamental results on simplices in Euclidean space.

\begin{definition}
    A simplex is a set in Euclidean space  \( \mathbb{R}^n \) such that two distinct vectors have the same distance.
\end{definition}

A simplex in \( \mathbb{R}^n \) of maximum size \( n+1 \) is unique up to isometry and scalar scaling. Examples include the equilateral triangles in dimension \( 2 \) and the regular tetrahedron in dimension \( 3 \). A maximum simplex can be inscribed in the unit sphere, and when this occurs, its centroid coincides with the origin of \( \mathbb{R}^n \). In this case, the inner product between distinct points of the simplex in the unit sphere is \( -\frac{1}{n} \). In this section, we provide proofs of certain intermediate results on simplices to make this paper self-contained.

\begin{proposition}
\label{proposition_simplex_n+1}
    There exists a simplex of size \( n+1 \) in \( \mathbb{R}^n \).
\end{proposition}

\begin{proof}
   We construct such a simplex as follows. Let \( e_i \in \mathbb{R}^n \) denote the standard basis vector with one in the \( i \)-th position and zeros elsewhere. The set \( \{e_1, e_2, \ldots, e_n\} \subseteq \mathbb{R}^n \) forms a simplex of size \( n \), with each pair of vectors separated by a distance of \( \sqrt{2} \).

    Define 
    \[
    v = \lambda (e_1 + e_2 + \cdots + e_n) \in \mathbb{R}^n
    \]
    for some \( \lambda \in \mathbb{R} \). The distance between \( v \) and each \( e_i \) is given by:
    \begin{align*}
    \|v - e_i\| &= \sqrt{\langle v - e_i, v - e_i \rangle} \\
    &= \sqrt{\|v\|^2 - 2\langle v, e_i \rangle + \|e_i\|^2} \\
    &= \sqrt{n\lambda^2 - 2\lambda + 1}.
    \end{align*}
    To ensure that \( \{v, e_1, e_2, \ldots, e_n\} \) forms a simplex, we require:
    \[
    n\lambda^2 - 2\lambda + 1 = 2.
    \]
    Solving this quadratic equation yields \( \lambda = \frac{1 \pm \sqrt{n + 1}}{n} \). Both values of \( \lambda \) result in a simplex of size \( n + 1 \).
\end{proof}

We denote the unit sphere in \( \mathbb{R}^n \) centered at the origin by \( \mathbb{S}^{n-1} \).

\begin{proposition}
\label{proposition_simplex_on_the_unit_sphere}
    A simplex in \( \mathbb{R}^n \) can be inscribed in \( \mathbb{S}^{n-1} \) by translating its centroid to the origin and applying an appropriate scaling factor.
\end{proposition}

\begin{proof}
    Assume \( S = \{v_1, v_2, \ldots, v_m\} \subseteq \mathbb{R}^n \) is a simplex with distance \( d \). Let 
    \[
    \bar{v} = \frac{1}{m} \sum_{i = 1}^{m} v_i
    \]
    be its centroid. We compute:
    \begin{align*}
        \|\bar{v} - v_j\|^2 &= \left\| \frac{1}{m} \sum_{i\neq j}(v_i - v_j) \right\|^2 \\
        &= \frac{1}{m^2} \left( \sum_{i\neq j} \|v_i - v_j\|^2 + \sum_{\substack{i_1, i_2 \neq j \\ i_1 \neq i_2}} \langle v_{i_1} - v_j, v_{i_2} - v_j \rangle \right) \\
        &= \frac{1}{m^2} \left( (m-1)d^2 + (m-1)(m-2) \frac{d^2}{2} \right) = \frac{d^2(m-1)}{2m}.
    \end{align*}
    Since the distance between \( \bar{v} \) and each \( v_j \) is the same, all points of \( S \) lie on a sphere centered at \( \bar{v} \). Translating by \( -\bar{v} \) centers the simplex at the origin, and applying an appropriate scaling factor ensures that all points lie on \( \mathbb{S}^{n-1} \).
\end{proof}

Applying Proposition \ref{proposition_simplex_on_the_unit_sphere} to the construction in Proposition \ref{proposition_simplex_n+1}, we obtain a simplex of size \( n+1 \) in \( \mathbb{S}^{n-1} \).

\begin{proposition}
\label{proposition_inner_product_of_simplex}
    Let \( S \subseteq \mathbb{S}^{n-1} \) be a simplex of size \( m \) centered at the origin. Then the inner product \( \lambda \) between distinct vectors in \( S \) is given by
    \[
    \lambda = - \frac{1}{m-1}.
    \]
\end{proposition}
\begin{proof}
    Assume \( S = \{v_1,v_2,\ldots,v_m\} \). The squared distance
    \[
    \|v_i - v_j\|^2 = 2 - 2\langle v_i, v_j \rangle.
    \]
    Since \( S \) is a simplex, the inner product between distinct vectors takes a single value \( \lambda \), so
    \[
    \langle v_i, v_j \rangle = 
    \begin{cases}
        1, & \text{if } i = j, \\
        \lambda, & \text{otherwise.}
    \end{cases}
    \]
    Since \( \sum_{i = 1}^{m} v_i = 0 \), we obtain
    \[
    \begin{split}
    0 &=\left\langle \sum_{i = 1}^{m} v_i, \sum_{i = 1}^{m} v_i \right\rangle\\
      &=  \sum_{i = 1}^{m} 1 + \sum_{i\neq j} \lambda \\
      &= m + m(m-1) \lambda.
    \end{split}
    \]
    It follows that
    \[
        \lambda = -\frac{1}{m-1}.
    \]
\end{proof}

\section{On Heged\"{u}s' conjecture}
\label{section_proof_of_Hegedus_conjecture}

In this section, we prove the following result.
\begin{theorem}
\label{main_theorem}
    Let \( C \subseteq H_q^n \) be an equidistant code. Assume the distance within the code \( C \) satisfies \( d \neq \frac{(q-1)n+1}{q} \). Then
    \[
    |C| \leq n(q-1).
    \]
\end{theorem}

The key idea is to map a code onto the unit sphere while preserving the information about code distance. 
Let $A_q = \{v_0, v_1, \ldots, v_{q-1}\} \subseteq \mathbb{S}^{q-2}$ be a simplex of size $q$ centered at the origin, where each $v_i$ is represented as a row vector. 
By Propositions~\ref{proposition_simplex_n+1}-\ref{proposition_simplex_on_the_unit_sphere}, such a simplex exists. We define the mapping
\begin{align*}
    \theta: H_q^n &\rightarrow \mathbb{S}^{(q-1)n-1}\\
    a = [a_1\ a_2\ \cdots\ a_n] &\mapsto \frac{1}{\sqrt{n}}[v_{a_1}\ v_{a_2}\ \cdots\ v_{a_n}]
\end{align*}
where each symbol \( a_i \in [0,q-1] \) and \( v_{a_i} \) is the corresponding vector in \( A_q \).
For each \( a \in H_q^n \), we have
\[
\langle \theta(a),\theta(a)\rangle = \frac{1}{n}\sum_{i = 1}^n\langle v_{a_i},v_{a_i}\rangle = 1.
\]
Hence the map \( \theta \) is well-defined as the image of each vector has unit length. 
For distinct \( a, b \in H_q^n \), the inner product between their images
\begin{align*}
    \langle \theta(a),\theta(b)\rangle &= \frac{1}{n}\sum_{i = 1}^n\langle v_{a_i},v_{b_i}\rangle\\
    &= \frac{1}{n}\left(\sum_{i: a_i=b_i} 1 + \sum_{i: a_i\neq b_i} \left(-\frac{1}{q-1}\right) \right) \quad &\text{(by Proposition \ref{proposition_inner_product_of_simplex})}\\ 
    &= \frac{1}{n}\left((n - \mathrm{d_H}(a, b)) + \mathrm{d_H}(a, b)\left(-\frac{1}{q-1}\right) \right)\\
    &= 1 - \frac{\mathrm{d_H}(a, b) q}{n(q-1)}.
\end{align*}
Thus, \( \theta \) preserves the information about code distance. Before giving our first proof, we need an elementary result on determinant.

\begin{lemma}
\label{lemma_determinant}
     Let \( J_m \) be the \( m\times m \) all-one matrix and \( I_m \) the \( m\times m \) identity matrix. Let \( \theta,\gamma\in \mathbb{R} \) be fixed real numbers. Then,
     \[
     \det(\gamma I_m+\theta J_m ) = (\gamma +m\theta)\gamma^{m-1}.
     \]
\end{lemma}

\begin{proof}
    Note that \( I_m \) and \( J_m \) can be diagonalized under a common basis such that \( I_m \) remains unchanged and \( J_m \) becomes \( \operatorname{diag}(m,0,\ldots,0) \). Therefore, the matrix \( \theta J_m + \gamma I_m \) will be diagonalized under the same basis to be \( \operatorname{diag}(\gamma+m\theta, \gamma,\ldots,\gamma) \). Taking the determinants, we have
    \[
    \det(\gamma I_m+\theta J_m ) =\det(\operatorname{diag}(\gamma+m\theta, \gamma,\ldots,\gamma)) = (\gamma +m\theta)\gamma^{m-1}.
    \]
\end{proof}

\noindent\textbf{First proof of Theorem \ref{main_theorem}}. Assume \( C\subseteq H_q^n \) is an equidistant code of distance \( d \). Then \( |C|\leq n(q-1)+1 \) by Corollary \ref{corollar_upper_bound_for_q_ary_case}. Since the map \( \theta \) preserves the information about code distance, \( \theta(C) \) is a simplex in \( \mathbb{S}^{(q-1)n-1} \).  

Write \(m= |C| \) and \( \theta(C) = \{w_1,w_2,\ldots,w_m\} \), in which the vectors are viewed as rows. Define
\[
M = [w_1^T\ w_2^T\ \cdots\ w_m^T].
\]
Let \( N = M^TM \). Then \( \operatorname{rank}(N) \leq \operatorname{rank}(M) \leq (q-1)n \). Note that 
\begin{align*}
    N_{i,j} = 
    \begin{cases}
        1, & \text{if }i=j,\\
        1-\frac{dq}{n(q-1)}, &\text{otherwise}.
    \end{cases}
\end{align*}
Thus, we can write
\begin{align*}
    N &= I_m +(1-\frac{dq}{n(q-1)})(J_m-I_m) \\
    &= \frac{dq}{n(q-1)}I_m+(1-\frac{dq}{n(q-1)})J_m.
\end{align*}

Suppose \( m = (q-1)n+1 \). Then \( \det(N)=0 \), since otherwise \( N \) is of full rank and thus \( \operatorname{rank}(N) = m = (q-1)n+1 \leq (q-1)n \), which is a contradiction. By Lemma \ref{lemma_determinant}, we have
\[
\det(N) = \left(\frac{dq}{n(q-1)}+m(1-\frac{dq}{n(q-1)})\right)\left(\frac{dq}{n(q-1)}\right)^{m-1}=0.
\]
Thus,
\[
\frac{dq}{n(q-1)}+m(1-\frac{dq}{n(q-1)})=0.
\]
Taking \( m = (q-1)n+1 \), we solve 
\[
d= \frac{(q-1)n+1}{q}.
\]
It follows that if $d\neq \frac{(q-1)n+1}{q}$, then $|C|$ cannot equal $(q-1)n+1$, implying that $|C|\le (q-1)n$. \hfill\qed

The following proposition is required in the second proof of Theorem \ref{main_theorem}, which we state without proof.

\begin{proposition}
\label{proposition_centered_at_the_origin}
    Let \( A_{n+1} \subseteq \mathbb{S}^{n-1} \) be a simplex of size \( n+1 \). Then \( A_{n+1} \) is centered at the origin.
\end{proposition}

\noindent\textbf{Second proof of Theorem \ref{main_theorem}.} Assume \( C \subseteq H_q^n \) is an equidistant code of distance \( d \). Then \( |C| \leq n(q-1) + 1 \) by Corollary \ref{corollar_upper_bound_for_q_ary_case}. As in the first proof, we apply \( \theta \) to \( C \) and obtain a simplex \( \theta(C) \subseteq \mathbb{S}^{(q-1)n-1} \). If \( |C| = (q-1)n + 1 \), then by Proposition \ref{proposition_centered_at_the_origin}, \( \theta(C) \) is centered at the origin, and by Proposition \ref{proposition_inner_product_of_simplex}, it must hold that
\[
\frac{1}{n} \left( n - \frac{d q}{q-1} \right) = -\frac{1}{(q-1)n},
\]
from which we obtain \( d = \frac{n(q-1) + 1}{q} \). Hence, if \( d \neq \frac{n(q-1) + 1}{q} \), we have \( |C| \leq (q-1)n \). \hfill \(\qed\)

When \( d = \frac{n(q-1) + 1}{q} \), there exist equidistant codes \( C \subseteq H_q^n \) with cardinality achieving \( n(q-1) + 1 \) if \( q \) is a prime power, by the notion of finite geometry.

\begin{example}
\label{example_equidistant_code_maximum_size}
    Let \( q \) be a prime power and \( \mathbb{F}_q \) the finite field of order \( q \). Let \( k \) be a positive integer. Consider all \( \frac{q^k - 1}{q-1} \) one-dimensional vector subspaces in \( \mathbb{F}_q^k \). Choose one nonzero vector from each subspace, and define a matrix \( G \) with those chosen vectors as columns, say
    \[
    G = [v_1\ v_2\ \cdots\ v_{\frac{q^k - 1}{q-1}}].
    \]
    It can be verified that \( G \) generates an \( \left[ n = \frac{q^k - 1}{q-1}, k, d = q^{k-1} \right]_q \) equidistant linear code \( C \). 
    Note that \( d = q^{k-1} = \frac{n(q-1) + 1}{q} \), since \( n = \frac{q^k - 1}{q-1} \). Also, \( |C| = q^k = n(q-1) + 1 \). Thus, when \( d = \frac{n(q-1) + 1}{q} \), the bound \( |C| \leq n(q-1) + 1 \) is tight.
\end{example}

\section{Generalization of Deza's theorem}
\label{section_generalization_of_Deza_theorem}

 Let \( x \in H_2^n \) be a binary word. We identify $x$ with its support set \( \mathrm{supp}(x) \subseteq [n] \). For example, if \( A \subseteq [n] \), then \( A \subseteq x \) means \( A \subseteq \mathrm{supp}(x) \). For \( y \in H_2^n \), the intersection \( x \cap y \) represents the intersection \( \mathrm{supp}(x) \cap \mathrm{supp}(y) \), and \( |x \cap y| \) is defined as \( |\mathrm{supp}(x) \cap \mathrm{supp}(y)| \).

Let $l \leq k\le n$ be non-negative integers. If $A \subseteq [n]$ has cardinality $k$, we call it a \textit{$k$-subset} of $[n]$. An \textit{$(n, k, l)$-family} is a collection of $k$-subsets of $[n]$ such that the intersection of any two distinct sets has cardinality exactly $l$.

A family $\mathcal{F} = \{B_1, B_2, \ldots, B_m\}$ of sets is called a \textit{$\Delta$-system} of cardinality $m$ if there exists a set $D \subseteq B_i$ for $i = 1, 2, \ldots, m$, such that the sets $B_1 \backslash D, B_2 \backslash D, \ldots, B_m \backslash D$ are pairwise disjoint. The set $D$ is referred to as the \textit{kernel} of the $\Delta$-system. A \textit{$\Delta(n, k, l)$-system} is a $\Delta$-system consisting of $k$-subsets of $[n]$ with a kernel of size $l$.
It is clear that a $\Delta(n, k, l)$-system is an $(n, k, l)$-family. The maximum cardinality of a $\Delta(n, k, l)$-system is given by
\[
\left\lfloor \frac{n - l}{k - l} \right\rfloor,
\]
since the $k - l$ elements in each $B_i \backslash D$ must be disjoint across all sets in the family, and there are $n - l$ elements available in $[n] \backslash D$. 

In 1973, Deza derived the following result on $(n,k,l)$-family \cite{Deza1973}.

\begin{theorem}[Deza \cite{Deza1973}, 1973]
\label{Theorem_deza_Delta_system_orginal_version}
    An $(n, k, l)$-family $\mathcal{F}$ with cardinality greater than $k^2 - k + 1$ is a $\Delta(n,k,l)$-system, and thus
    \[
    |\mathcal{F}| \leq \frac{n - l}{k - l}.
    \]
\end{theorem}

Later, in \cite{DezaFrankl1981}, Deza and Frankl pointed out that a slight modification of the argument in \cite{Deza1973} gives the following better result.

\begin{theorem}
\label{Theorem_deza_Delta_system}
    If an $(n, k, l)$-family $\mathcal{F}$ has cardinality 
    \[|\mathcal{F}|>\max\{l+2,(k-l)^2+(k-l)+1\},\]
    then $\mathcal{F}$ is a $\Delta(n,k,l)$-system, and thus
    \[
    |\mathcal{F}| \leq \frac{n - l}{k - l}.
    \]
\end{theorem}

This result immediately implies the following upper bound for binary equidistant codes, which also appears in \cite{Deza1973} under a stricter condition on the cardinality of $C$. For binary case, we only concern equidistant codes whose distances are even, since an binary equidistant code with odd distance contains at most two codewords.

\begin{corollary}
\label{corollary_distant_dependent_bound_binary_case}
Let $C \subseteq H_2^n$ be an equidistant code with even distance $d$.
\begin{enumerate}
    \item[(a)] If $C$ contains the all-zero codeword and $|C| > \frac{d^2 + 2d + 8}{4}$, then $C \backslash \{0\}$ is a $\Delta(n, d, d/2)$-system.
    \item[(b)] We have 
    \begin{equation}
    \label{inequality_equidistant_code_upper_bound}
        |C| \leq \max\left\{\frac{d^2 + 2d + 8}{4},\left\lfloor \frac{2n}{d} \right\rfloor\right\}.
    \end{equation}
    Particularly, if $n \geq \frac{d(d^2 + 2d + 8)}{8}$, then
    \[
    |C| \leq \left\lfloor \frac{2n}{d} \right\rfloor.
    \]
\end{enumerate}
\end{corollary}

\begin{proof}
\begin{enumerate}
    \item[(a)] Since $C$ contains the all-zero codeword, $C \backslash \{0\} \subseteq J_2^{n,d}$ is an $(n, d, d/2)$-family. Each codeword in $C \backslash \{0\}$ has weight $d$ (since the distance to the all-zero codeword is $d$), and the distance $d$ between any two distinct codewords implies an intersection size of $d/2$. If $|C| > \frac{d^2 + 2d + 8}{4}$, then $|C \backslash \{0\}| = |C| - 1 > \frac{d^2 + 2d + 4}{4} = \left(\frac{d}{2}\right)^2 + \frac{d}{2} + 1 \geq \frac{d}{2} + 2$. By Theorem \ref{Theorem_deza_Delta_system}, $C \backslash \{0\}$ is a $\Delta(n, d, d/2)$-system.
    
    \item[(b)] Without loss of generality, assume $C$ contains the all-zero codeword. If $|C| \le \frac{d^2 + 2d + 8}{4}$, then \eqref{inequality_equidistant_code_upper_bound} follows. If $|C| > \frac{d^2 + 2d + 8}{4}$, then by (a), $C \backslash \{0\}$ is a $\Delta(n, d, d/2)$-system. Thus,
    \[
    |C \backslash \{0\}| \leq \frac{n - d/2}{d - d/2} = \frac{2n}{d} - 1.
    \]
    Therefore, $|C| = |C \backslash \{0\}| + 1 \leq \frac{2n}{d}$. This proves \eqref{inequality_equidistant_code_upper_bound}.
    
    If $n \geq \frac{d(d^2 + 2d + 8)}{8}$, then $\frac{2n}{d} \geq \frac{d^2 + 2d + 8}{4}$. Thus, the bound simplifies to $|C| \leq \left\lfloor \frac{2n}{d} \right\rfloor$.
\end{enumerate}
\end{proof}

Next, we will extend Deza's theorem (Theorem \ref{Theorem_deza_Delta_system}) on $\Delta$-systems from the binary case to the $q$-ary case and then apply this generalization to $q$-ary equidistant codes.

Let $x, y \in J^{n,k}$. Regarding $x$ and $y$ as sets, we have
\[ |x \cap y| = k - \frac{\mathrm{d_H}(x,y)}{2}. \]
We extend this to $q$-ary Johnson spaces by defining the intersection size of two words $x, y \in J_q^{n,k}$ as
\[ |x \cap y|_q = k - \frac{\mathrm{d_H}(x,y)}{2}, \]
without explicitly specifying their intersection. Since $\mathrm{d_H}(x,y) \leq \mathrm{wt}(x) + \mathrm{wt}(y) = 2k$, the intersection size $|x \cap y|_q$ is always non-negative. 
However, $|x \cap y|_q$ may be a half-integer whenever $\mathrm{d_H}(x,y)$ is odd.

Assume $C\subseteq J_q^{n,k}$ is an equidistant code with distance $d$. Then by the above definition, $C$ is a code with a single intersection size  $l = k-d/2<k$. Thus, we also refer to the equidistant code \( C \) as an \textit{$(n,k,l)_q$-family}. Deza's theorem states that if an $(n,k,l)$-family $\mathcal{F}$ has large cardinality, then $\mathcal{F}$ is a $\Delta$-system. Below, we extend the concept of a \(\Delta\)-system from the binary case to the \(q\)-ary case.

\begin{definition}
A code \( C \subseteq J_q^{n,k} \) is a \textbf{\( \Delta_q(n, k, l) \)-system} with kernel \( D \subseteq [n] \), where \( |D| = l \), if 
\begin{enumerate}
    \item[(1)] for any two distinct codewords \( x, y \in C \) it holds that \( x(i) = y(i) \neq 0 \) for all \( i \in D \), and
    \item[(2)] the sets $\mathrm{supp}(x)\backslash D, x\in C$ are pairwise disjoint.
\end{enumerate}
Note that for \( q = 2 \), a \( \Delta_2(n, k, l) \)-system coincides with a \( \Delta(n, k, l) \)-system defined before.
\end{definition}

In Theorem \ref{theorem_q_ary_version_of_Deza_theorem}, we will prove that if the size of an $(n,k,l)_q$-family $\mathcal{F}$ is sufficiently large, it must form a $\Delta_q(n,k,l)$-system. The proof employs a technique that transforms a $q$-ary code into a binary code while preserving the intersection size information.

To transform a $q$-ary code into a binary code, each $q$-ary symbol is replaced by a suitably chosen binary vector of length $q$.
Define 
\begin{align*}
    \begin{split}
        w_0 &= [0, 0, \ldots, 0, 0],\\
        w_1 &= [1, 0, \ldots, 0, 1],\\
        w_2 &= [0, 1, \ldots, 0, 1],\\
        &\cdots\\
        w_{q-1} &= [0, 0, \ldots, 1, 1].
    \end{split}
\end{align*}
The set \( \mathcal{Q}_q = \{w_0, w_1, \ldots, w_{q-1}\} \) forms an equidistant code with distance 2 and satisfies the following properties:
\begin{enumerate}
    \item[(1)] \( w_i \neq 0\) if and only if \( i \neq 0 \), and
    \item[(2)] \( |w_i \cap w_j| = \begin{cases} 
    0, & \text{if } i = 0 \text{ or } j = 0, \\ 
    1, & \text{if } i, j \neq 0 \text{ and } i \neq j, \\ 
    2, & \text{if } i = j \neq 0.
    \end{cases} \)
\end{enumerate}

Define the mapping
\[
\psi: H_q^n \to H_2^{q n}, \quad a = [a(1), a(2), \cdots, a(n)] \mapsto [w_{a(1)}, w_{a(2)}, \cdots, w_{a(n)}].
\]
The map $\psi$ scales weight and distance as follows: $\mathrm{wt}(\psi(x)) = 2 \cdot \mathrm{wt}(x)$ and $\mathrm{d_H}(\psi(x), \psi(y)) = 2 \cdot \mathrm{d_H}(x, y)$. 
Consequently, $\psi$ transforms an $(n,k,l)_q$-family into a $(qn, 2k, 2l)$-family. 
We partition the coordinate set $[qn]$ into $n$ intervals as 
\[
[qn] = \bigcup_{i=1}^n I_i,
\]
where $I_i := [q(i-1)+1, qi]$ for $i = 1, 2, \ldots, n$. 

For a subset $S \subseteq [n]$, we define a projection that maps a subset $A \subseteq [n]$ to its intersection with $S$, given by
\[
\pi_S: 2^{[n]} \to 2^{[n]}, \quad A \mapsto A \cap S.
\]
This projection extends naturally to the binary Hamming space as follows:
\[
\pi_S: H_2^n \to H_2^n, \quad x \mapsto y,
\]
where 
\[ y(i) = 
\begin{cases}
x(i),& \text{ if $i\in S$},\\
0,& \text{ otherwise}.
\end{cases}
\]

\begin{theorem}
\label{theorem_q_ary_version_of_Deza_theorem}
Let $\mathcal{F}$ be an $(n,k,l)_q$-family, where $l = k - d/2$ and $0 < d \leq 2k$. If
\[
|\mathcal{F}| > \max\{2l + 2, 4(k - l)^2 + 2(k - l) + 1, q - 1\},
\]
then $l$ is an integer, $\mathcal{F}$ is a $\Delta_q(n, k, l)$-system, and
\[
|\mathcal{F}| \leq \frac{n - l}{k - l}.
\]
\end{theorem}

\begin{proof}
We utilize the map $\psi$ and the partition $\{I_i : i \in [n]\}$ of $[qn]$ defined previously. Applying $\psi$ to $\mathcal{F}$, we obtain a $(qn, 2k, 2l)$-family $\psi(\mathcal{F})$. Since $\psi$ is injective, $|\psi(\mathcal{F})| = |\mathcal{F}|$. If $|\mathcal{F}| > \max\{2l + 2, 4(k - l)^2 + 2(k - l) + 1\}$, then by Theorem~\ref{Theorem_deza_Delta_system}, $\psi(\mathcal{F})$ is a $\Delta(qn, 2k, 2l)$-system with some kernel $D \subseteq [qn]$ of size $|D| = 2l$.

For each $i \in [n]$, consider the projection of $\psi(\mathcal{F})$ onto $I_i$:
\[
\begin{aligned}
\pi_{I_i}(\psi(\mathcal{F})) 
&= \{\pi_{I_i}([w_{a(1)}, \ldots, w_{a(i-1)}, w_{a(i)}, w_{a(i+1)}, \ldots, w_{a(n)}]) : a \in \mathcal{F}\} \\
&= \{[0,\ldots,0,w_{a(i)},0,\ldots,0] : a \in \mathcal{F}\},
\end{aligned}
\]
where $w_{a(i)}$ is the $i$-th block of $\psi(a)$, corresponding to the coordinates in $I_i$. For any two distinct codewords $x, y \in \mathcal{F}$, since $\psi(\mathcal{F})$ is a $\Delta$-system with kernel $D$, we have
\[
D = \psi(x) \cap \psi(y).
\]
Thus,
\begin{equation}
\label{equality_of_local_kernel}
\pi_{I_i}(\psi(x)) \cap \pi_{I_i}(\psi(y)) = (\psi(x) \cap I_i) \cap (\psi(y) \cap I_i) = (\psi(x) \cap \psi(y)) \cap I_i = D \cap I_i.
\end{equation}

\begin{claim}
\label{claim_size_of_Ii_cap_D}
For each $i \in [n]$, either $|I_i \cap D| = 0$ or $|I_i \cap D| = 2$.
\end{claim}

\begin{proof}[Proof of Claim \ref{claim_size_of_Ii_cap_D}]
Let $x, y \in \mathcal{F}$ be any two distinct codewords. By \eqref{equality_of_local_kernel}, we have
\[
|I_i \cap D| = |\pi_{I_i}(\psi(x)) \cap \pi_{I_i}(\psi(y))|=|w_{x(i)} \cap w_{y(i)}| \leq 2,
\]
by the basic properties of the vectors $w_i$. We now rule out the case $|I_i \cap D| = 1$. Suppose $|I_i \cap D| = 1$ for some $i \in [n]$. Then, by \eqref{equality_of_local_kernel},
\[
|\pi_{I_i}(\psi(x)) \cap \pi_{I_i}(\psi(y))| = |I_i \cap D| = 1
\]
for any distinct $x, y \in \mathcal{F}$, which further implies that $\pi_{I_i}(\psi(x))$ and $\pi_{I_i}(\psi(y))$ are non-zero and distinct. Therefore, 
\[|\mathcal{F}|=|\{\pi_{I_i}(\psi(x)) : x \in \mathcal{F}\}|=|\{[0,\ldots,0,w_{a(i)},0,\ldots,0] : a \in \mathcal{F}\}| \leq q - 1.\]
However, this contradicts to the assumption that $|\mathcal{F}| > q - 1$. Hence, $|I_i \cap D| \neq 1$, and the claim follows.
\end{proof}

Define $K \subseteq [n]$ as the set of indices $i$ where $I_i$ intersects $D$ non-trivially, namely,
\[
K = \{i \in [n] : |I_i \cap D| = 2\}.
\]
We now show that $\mathcal{F}$ is a $\Delta_q(n, k, |K|)$-system with kernel $K$. For distinct codewords $x, y \in \mathcal{F}$:
\begin{itemize}
    \item If $i \in K$, then by \eqref{equality_of_local_kernel}, $|\pi_{I_i}(\psi(x)) \cap \pi_{I_i}(\psi(y))| = |I_i \cap D| = 2$. Thus, $|w_{x(i)} \cap w_{y(i)}| = 2$, implying $x(i) = y(i) \neq 0$.
    \item If $i \in [n] \backslash K$, then $|I_i \cap D| = 0$, so $|\pi_{I_i}(\psi(x)) \cap \pi_{I_i}(\psi(y))| = 0$. Thus, $|w_{x(i)} \cap w_{y(i)}| = 0$, implying either $x(i) = 0$ or $y(i) = 0$.
\end{itemize}

Finally, we compute
\[
|D| = \left|\bigcup_{j \in K} (D \cap I_j)\right| = \sum_{j \in K} |D \cap I_j| = \sum_{j \in K} 2 = 2|K|.
\]
Since $|D| = 2l$, we have $|K| = l$. Note that this also means that $l$ is an integer. Thus, $\mathcal{F}$ is a $\Delta_q(n, k, l)$-system with kernel $K$ of size $l$. The bound $|\mathcal{F}| \leq \frac{n - l}{k - l}$ follows from the properties of a $\Delta_q(n, k, l)$-system.
\end{proof}

\begin{corollary}
\label{corollary_bound_for_q_ary_equidistant_code}
Let $C \subseteq H_q^n$ be an equidistant code with distance $d$.
\begin{itemize}
    \item[(a)] If $C$ contains the all-zero codeword and $|C| > \max\{d^2 + d + 2, q\}$, then $d$ is even, and $C \backslash \{0\}$ is a $\Delta_q(n, d, d/2)$-system.
    \item[(b)] We have
    \begin{equation}
    \label{inequality_q_ary_equidistant_code_upper_bound}
        |C| \leq \max\left\{d^2 + d + 2, q,\left\lfloor \frac{2n}{d} \right\rfloor\right\}.
    \end{equation}
    Particularly, if $n \geq \max\left\{\frac{d(d^2 + d + 2)}{2}, \frac{dq}{2}\right\}$, then
    \[
    |C| \leq \left\lfloor \frac{2n}{d} \right\rfloor.
    \]
\end{itemize}
\end{corollary}

\begin{proof}
\begin{itemize}
    \item[(a)] Assume $C \subseteq H_q^n$ is an equidistant code containing the all-zero codeword. Then $C \backslash \{0\} \subseteq J_q^{n,d}$ is an $(n,d,d/2)_q$-family. If $|C| > \max\{d^2 + d + 2, q\}$, then $|C \backslash \{0\}| > \max\{d+2, d^2 + d + 1, q-1\}$. Applying Theorem \ref{theorem_q_ary_version_of_Deza_theorem} to $C \backslash \{0\}$ with $k = d$ and $l = d/2$, we conclude that $l$ is an integer, $d = 2l$ is even, and $C \backslash \{0\}$ is a $\Delta_q(n,d,d/2)$-system. This proves (a).
    
    \item[(b)] Without loss of generality, assume $C$ contains the all-zero codeword. If $|C| \le \max\{d^2 + d + 2, q\}$, then \eqref{inequality_q_ary_equidistant_code_upper_bound} follows. If $|C| > \max\{d^2 + d + 2, q\}$, then by (a), $C \backslash \{0\}$ is a $\Delta_q(n,d,d/2)$-system. Thus,
    \[
    |C \backslash \{0\}| \leq \frac{n - d/2}{d - d/2} = \frac{2n}{d} - 1,
    \]
    which implies $|C| = |C \backslash \{0\}| + 1 \leq \frac{2n}{d}$. This proves \eqref{inequality_q_ary_equidistant_code_upper_bound}.
    
    If $n \geq \max\left\{\frac{d(d^2 + d + 2)}{2}, \frac{dq}{2}\right\}$, then $\frac{2n}{d} \geq \max\{d^2 + d + 2, q\}$. Thus, the bound simplifies to $|C| \leq \left\lfloor \frac{2n}{d} \right\rfloor$.
\end{itemize}
\end{proof}

The bound $|C| \leq \left\lfloor \frac{2n}{d} \right\rfloor$ is tight for large $n$, as shown by the optimal $\Delta_q(n, d/2, 0)$-system construction.

\section{Discussions}
\label{section_discussions}

In this section, we propose several open questions for further explorations.

\begin{question}
Are there other equidistant codes $C \subseteq H_q^n$ that attain the upper bound $(q-1)n + 1$, except those presented in Example \ref{example_equidistant_code_maximum_size}?
\end{question}

\begin{question}
Denote $f_q(k,l)$ the smallest integer $m$ such that any $(n,k,l)_q$-family of cardinality greater than $m$ forms a $\Delta_q(n,k,l)$-system. By Theorem~\ref{theorem_q_ary_version_of_Deza_theorem}, we know 
\[f_q(k,l)\le  4(k - l)^2 + 2(k - l) + 1,\]
when $\max\{2l + 2, q - 1\}\le  4(k - l)^2 + 2(k - l) + 1$. What is the exact value of $f_q(k,l)$?
\end{question}





\section*{Acknowledgment}
The authors thank Professor Chong Shangguan for informing us Deza’s theorem and for the valuable discussions. Sihuang Hu and Hexiang Huang are partially funded by National Key R\&D Program of China under Grant No. 2021YFA1001000, National Natural Science Foundation of China under Grant No. 12231014. Wei-Hsuan Yu is supported by NSTC of Taiwan, 113-2115-M-008-007-MY2.


\bibliographystyle{plain}
\bibliography{references}

\begin{thebibliography}{99}

\bibitem{delsarte1973algebraic}
Philippe Delsarte, \emph{An Algebraic Approach to the Association Schemes of Coding Theory}, Philips Research Reports, Suppl. 10:1--97, 1973.

\bibitem{delsarte1975association}
Philippe Delsarte, \emph{The Association Schemes of Coding Theory}, in \emph{Combinatorics: Proceedings of the NATO Advanced Study Institute held at Nijenrode Castle, Breukelen}, pages 143--161, Springer Netherlands, 1975.

\bibitem{hegedus_2024_a_new_upper}
G{\'a}bor Heged{\"u}s, \emph{A New Upper Bound for Codes with a Single Hamming Distance}, 2024, \url{https://arxiv.org/abs/2409.07877}, arXiv:2409.07877 [cs.IT].

\bibitem{barga_musin_2011_bounds}
Alexander Barg and Oleg R. Musin, \emph{Bounds on Sets with Few Distances}, Journal of Combinatorial Theory, Series A, 118(4):1465--1474, 2011.

\bibitem{babai_frankl_1992_linear}
L{\'a}szl{\'o} Babai and P{\'e}ter Frankl, \emph{Linear Algebra Methods in Combinatorics}, 1992, Unpublished manuscript.

\bibitem{Babai_Snevily_Wilson_1995_A_new_proof}
L{\'a}szl{\'o} Babai, Hunter Snevily, and Richard M. Wilson, \emph{A New Proof of Several Inequalities on Codes and Sets}, Journal of Combinatorial Theory, Series A, 71(1):146--153, 1995.

\bibitem{Deza1974}
Michel M. Deza, \emph{Solution d'un probl{\`e}me de Erd{\H{o}}s-Lov{\'a}sz}, Journal of Combinatorial Theory, Series B, 16(2):166--167, 1974.

\bibitem{DezaErdosFrankl1978}
Michel M. Deza, Paul Erd{\H{o}}s, and P{\'e}ter Frankl, \emph{Intersection Properties of Systems of Finite Sets}, Proceedings of the London Mathematical Society, s3-36(2):369--384, 1978.

\bibitem{MacWilliams1977}
Florence J. MacWilliams and Neil J. A. Sloane, \emph{The Theory of Error-Correcting Codes}, North-Holland Publishing Company, Amsterdam, 1977.

\bibitem{Evangelaras2003ApplicationsOH}
Haralambos Evangelaras, Christos Koukouvios, and Jennifer Seberry, \emph{Applications of Hadamard matrices}, Journal of Telecommunications and Information Technology, 2003, \url{https://api.semanticscholar.org/CorpusID:15284787}.

\bibitem{Deza1973}
Michel M. Deza, \emph{Une propri{\'e}t{\'e} extr{\'e}male des plans projectifs finis dans une classe de codes {\'e}quidistants}, Discrete Mathematics, 6(4):343--352, 1973.

\bibitem{EtzionRaviv2015}
Tuvi Etzion and Netanel Raviv, \emph{Equidistant codes in the Grassmannian}, Discrete Applied Mathematics, 186:87--97, 2015.

\bibitem{GorlaRavagnani2016}
Elisa Gorla and Alberto Ravagnani, \emph{Equidistant subspace codes}, Linear Algebra and its Applications, 490:48--65, 2016.

\bibitem{DezaFrankl1981}
Michel M. Deza and P{\'e}ter Frankl, \emph{Every large set of equidistant (0, +1, {-}1)-vectors forms a sunflower}, Combinatorica, 1:225--231, 1981.

\bibitem{ZinovievTodorov2007}
Victor A. Zinov'ev and Todor J. Todorov, \emph{On the construction of $q$-ary equidistant codes}, Problems of Information Transmission, 43(4):302--308, 2007.

\bibitem{FU2003157}
Fang-Wei Fu, Torleiv Kløve, Yuan Luo, and Victor K. Wei, \emph{On equidistant constant weight codes}, Discrete Applied Mathematics, 128(1):157--164, 2003.

\end{thebibliography}

{\small
\noindent State Key Laboratory of Cryptography and Digital Economy Security, Shandong University, Qingdao 266237, China \\
Key Laboratory of Cryptologic Technology and Information Security, Ministry of Education, Shandong University, Qingdao 266237, China \\
School of Cyber Science and Technology, Shandong University, Qingdao 266237, China \\
Quan Cheng Laboratory, Jinan 250103, China

\noindent\text{Email address:} {\href{mailto:husihuang@sdu.edu.cn}{\textcolor{black}{husihuang@sdu.edu.cn}}}

\vspace{0.2cm}
\noindent School of Cyber Science and Technology, Shandong
University, Qingdao 266237, China 

\noindent\text{Email address:} {\href{mailto:hexianghuang@foxmail.com}{\textcolor{black}{hexianghuang@foxmail.com}}}

\vspace{0.2cm}

\noindent Department of Mathematics, National Central University, Taoyuan 32001, Taiwan

\noindent\text{Email address:} {\href{mailto:u690604@gmail.com}
{\textcolor{black}{u690604@gmail.com}}}

During the preparation of this work the authors used Grok AI in order to improve the readability and language of the manuscript. After using this tool, the authors reviewed and edited the content as needed and take full responsibility for the content of the published article.
\end{document}